\documentclass[10pt]{amsart}

\usepackage{amsfonts}
\usepackage{amssymb}
\usepackage{amsmath, amsthm, latexsym}
\usepackage[all]{xy}
\usepackage{hyperref}
\usepackage{graphicx}
\usepackage[margin=1.5in]{geometry}
\usepackage{color}

\def \C{\mathbb{C}}
\def \Z{\mathbb{Z}}
\def \R{\mathbb{R}}

\def \Q{\mathbb{Q}}
\def \GL{\textup{GL}}
\def \SL{\textup{SL}}
\def \K{\mathcal{K}}
\def \Log{\operatorname{Log}}
\def \val{\operatorname{val}}
\def \trop{\operatorname{trop}}

\def \Mat{\operatorname{Mat}}
\def \GL{\operatorname{GL}}
\def \diag{\operatorname{diag}}

\theoremstyle{plain}
\newtheorem{Th}{Theorem}[section]
\newtheorem{Lem}[Th]{Lemma}

\newtheorem{Conj}[Th]{Conjecture}

\theoremstyle{definition}

\newtheorem{Def}[Th]{Definition}

\begin{document}

\title{Invariant factors as limit of singular values of a matrix}
\author{Kiumars Kaveh}
\address{Department of Mathematics, University of Pittsburgh,
Pittsburgh, PA, USA.}
\email{kaveh@pitt.edu}

\author{Peter Makhnatch}
\address{Department of Mathematics, University of Pittsburgh,
Pittsburgh, PA, USA.}
\email{pem43@pitt.edu}

\begin{abstract}
The paper concerns a result in linear algebra motivated by ideas from tropical geometry. Let $A(t)$ be an $n \times n$ matrix whose entries are Laurent series in $t$. We show that, as $t \to 0$, the logarithms of singular values of $A(t)$ approach the invariant factors of $A(t)$. This leads us to suggest logarithms of singular values of an $n \times n$ complex matrix as an analogue of the logarithm map on $(\C^*)^n$ for the matrix group $\GL(n, \C)$. 
\end{abstract}

\thanks{The first author was partially supported by a National Science Foundation Grant (Grant ID: DMS-1601303).}

\date{\today}
\maketitle

\section{Introduction}
This paper is intended for general mathematics audience and hopefully for the most part is accessible to interested mathematics students.
 
Let $A \in \Mat(n, \C)$ be an $n \times n$ complex matrix. Recall that the singular values of $A$ are square roots of the eigenvalues of the Hermitian matrix $AA^*$, where $A^* = \bar{A}^T$ is the conjugate transpose of $A$ (note that the eigenvalues of $AA^*$ are non-negative real numbers). The \emph{singular value decomposition} (for square matrices) states that there are $n \times n$ unitary matrices $U$, $W$ such that $A = UDW$, where $D$ is a diagonal matrix with singular values of $A$ as its diagonal entries. 

There is a non-Archimedean analogue of the singular value 
decomposition usually referred to as the \emph{Smith normal form} theorem. Here one replaces $\C$ with the field of fractions of a principal ideal domain (see Theorem \ref{th-Smith-normal-form}). We will be interested in the case of the field of formal Laurent series. Let $\mathcal{O} = \C[[t]] = \{ \sum_{i=0}^\infty c_i t^i \mid c_i \in \C \}$ be the algebra of formal power series in one variable $t$ and let $\mathcal{K} = \C((t)) = \{ \sum_{i=-k}^\infty c_i t^i \mid k \in \Z,~c_i \in \C\}$ be its field of fractions which is the field of formal Laurent series in $t$. Here the term \emph{formal} means that we do not require the series to be convergent and $\C[[t]]$ (respectively $\C((t))$) consists of all, possibly non-convergent, power series (respectively Laurent series).
Let $A(t)=(A_{ij}(t)) \in \Mat(n, \mathcal{K})$. 
The Smith normal form theorem in this case states that there are matrices $P(t), Q(t) \in \GL(n, \mathcal{O})$ and a diagonal matrix $\tau(t) \in \Mat(n, \mathcal{K})$ such that $A(t) = P(t) \tau(t) Q(t)$. The diagonal entries of $\tau(t)$ are usually called the \emph{invariant factors of $A(t)$}. They are unique up to multiplication by units in $\mathcal{O}$.
We recall that any power series with non-zero constant term is a unit in $\mathcal{O}$, and hence any Laurent series can be written as a power of $t$ times a unit in $\mathcal{O}$. Thus, we can moreover take the diagonal entries of $\tau(t)$ to be of the form $t^{v_1}, \ldots, t^{v_n}$, where $v_i \in \Z$. With slight abuse of terminology we refer to the $v_i$ also as the \emph{invariant factors of $A(t)$}. 

Let us assume that $A(t)$ is convergent for a punctured neighborhood of the origin, that is, there exists $\epsilon > 0$ such that all the Laurent series $A_{ij}(t)$ are convergent for any $0 < |t| < \epsilon$. In this paper we prove the following statement:

\begin{Th} \label{th-main}
For sufficiently small $t \neq 0$, let $d_1(t) \leq \cdots \leq d_n(t)$ denote the singular values of $A(t)$ ordered increasingly. Also let $v_1 \geq \cdots \geq v_n$ be the invariant factors of $A(t)$ ordered decreasingly. We then have:
$$\lim_{t \to 0^+} (\log_t(d_1(t)), \ldots, \log_t(d_n(t))) = (v_1, \ldots, v_n).$$
\end{Th}

We remark that in the usual statement of the Smith normal form (Theorem \ref{th-Smith-normal-form}) the invariant factors are (often) ordered increasingly, but in the above we are considering them ordered decreasingly.  

The motivation to consider this statement comes from tropical geometry and an attempt to generalize the notions of (Archimedean) amoeba and tropical variety (non-Archimedean amoeba) for subvarieties in $(\C^*)^n$, to subvarieties in $\GL(n, \C)$ (see Sections \ref{sec-tropical} and \ref{sec-amoeba}). This is related to tropical geometry on {\it spherical varieties} and {\it spherical tropicalization}. In Section \ref{sec-sph-trop} we give some background on spherical tropicalization. We give a proof of Theorem \ref{th-main} in Section \ref{sec-main}. Before that we briefly review motivating material from tropical geometry and we also recall some background material from linear algebra. 

In Section \ref{sec-amoeba} we state the conjecture that an analogue of Theorem \ref{th-Jonsson} (for the classical case of amoebae and tropical varieties) holds for subvarieties of $\GL(n, \C)$  (Conjecture \ref{conj-sph-amoeba}). That is, if $Y \subset \GL(n, \C)$ is a subvariety, then the set of logarithms of singular values of matrices lying on $Y(\overline{\K})$ approaches, in the sense of  Kuratowski, to the (topological) closure of the set of invariant factors of matrices in $Y(\overline{\K})$.

We point out that Theorem \ref{th-main} (without proof) and Conjecture \ref{conj-sph-amoeba} are also stated in \cite[Section 7, Example 7.7 and Proposition 7.8]{KM}.

\subsection*{Acknowledgement} This paper is the outcome of the summer undergraduate research project of Peter Makhnatch at the University of Pittsburgh (2018). We are thankful to the anonymous referees for  very careful reading of the paper and valuable suggestions that greatly improved the content and presentation of the paper.

\section{Tropicalization and non-Archimedean amoebae}   \label{sec-tropical}

From the point of view of algebraic geometry, tropical geometry is concerned with describing the ``(exponential) asymptotic behavior'', or ``(exponential) behavior at infinity'', of subvarieties in $(\C^*)^n$ where $\C^* = \C \setminus \{0\}$. With componentwise multiplication, $(\C^*)^n$ is an abelian group. It is usually referred to as an \emph{algebraic torus} and is one of the basic examples of algebraic groups. A subvariety of $(\C^*)^n$ is called 
a {\it very affine variety}. The behavior at infinity of a subvariety $Y \subset (\C^*)^n$ is encoded in a union of convex polyhedral cones called the {\it tropical variety} of $Y$, as we explain below. One of the early sources of the notion of tropical variety is the 1971 paper by George Bergman on the logarithmic limit-set of an algebraic variety (\cite{Bergman}). 


Let $\K = \C((t))$ be the field of formal Laurent series in one indeterminate $t$. 
Then $\overline{\K} = \C\{\{t\}\} =  \bigcup_{k=1}^\infty \C((t^{1/k}))$ is the field of formal Puiseux series. It is the algebraic closure of the field of Laurent series $\K$. This fact is the content of the Newton-Puiseux theorem (for a nice account of this see \cite[Lecture 12]{Abhyankar}). 
The field $\overline{\K}$ comes equipped with the {\it order of vanishing} valuation 
$\val: (\overline{\K})^* = \overline{\K} \setminus \{0\} \to \Q$ defined as follows: for a Puiseux series $f(t) = \sum_{i=m}^\infty a_i t^{i/k}$, where $a_m \neq 0$, we put $\val(f) = m/k$.
The valuation $\val$ gives rise to the {\it tropicalization map} $\trop$ from $(\overline{\K}^*)^n$ to $\Q^n$:
$$\trop(f_1, \ldots, f_n) = (\val(f_1), \ldots, \val(f_n)).$$
Let $Y \subset (\C^*)^n$ be a subvariety with ideal $I=I(Y) \subset \C[x_1^\pm, \ldots, x_n^\pm]$. Let $Y(\overline{\K})$ denote the Puiseux series valued points on $Y$, that is, $Y(\overline{\K}) = \{ f=(f_1, \ldots, f_n) \in (\overline{\K}^*)^n \mid p(f_1, \ldots, f_n) = 0,~\forall p \in I\}$. The {\it tropical variety of $Y$}, denoted $\trop(Y)$, is defined to be the closure (in $\R^n$) of the image of $Y(\overline{\K})$ under the map $\trop$. 
One shows that the tropical variety of a subvariety always has the structure of a \emph{fan} in $\R^n$, that is, it is a finite union of (strictly) convex rational polyhedral cones (we refer the interested reader to \cite[Chapter 3]{Maclagan-Sturmfels}).
We recall that a \emph{convex rational polyhedral cone} in $\R^n$ is the convex cone (with apex at the origin) generated by a finite number of vectors in $\Q^n$. It is called \emph{strictly convex} if it is a pointed cone, i.e. does not contain any lines through the origin. Finally, a \emph{fan} is a finite collection of strictly convex rational polyhedral cones that intersect on their common faces (see for example \cite[Section 2.3]{Maclagan-Sturmfels}). 

The set  $\trop(Y)$ describes all the exponential directions along which $Y$ approaches infinity. One can make this precise using the theory of \emph{toric varieties}. For people familiar with this theory we state the following (see \cite[Section 6.4]{Maclagan-Sturmfels} and the references therein): 

\begin{Th}
Let $\Sigma$ be a fan in $\R^n$. Then the closure $\overline{Y}$ of $Y$ in the toric variety $X_\Sigma \supset (\C^*)^n$ is a complete variety (equivalently, is compact in the usual Euclidean topology) if and only if the support of $\Sigma$ contains $\trop(Y)$, that is, the union of the cones in $\Sigma$ contains $\trop(Y)$. 
\end{Th}

The above theorem is usually referred to as the tropical compactification theorem. It appears in \cite{Tevelev} (and is also implicit in \cite[\S 3]{DP2}).  

Since intersection theoretic data does not change under deformations, tropical geometry can often be used to give combinatorial or piecewise linear formulae for intersection theoretic problems. Many developments and applications of tropical geometry, at least in algebraic geometry, come from this point of view.



\section{Spherical tropicalization}  \label{sec-sph-trop}

In \cite{Vogiannou}, Tevelev and Vogiannou far extend the notion of tropicalization map for the algebraic torus $(\C^*)^n$ to so-called \emph{spherical homogeneous spaces} (from the theory of algebraic groups). In this section, we give a very brief introduction to spherical varieties and their tropicalization map.  

For the following paragraph, we assume some familiarity with the theory of algebraic groups. At the end of this section, we give a description of the spherical tropicalization map for the case of the general linear group $\GL(n, \C)$ without any necessary knowledge of algebraic groups. 

One starts with a reductive algebraic group $G$ over $\C$. Common examples are $G = \GL(n, \C)$, $\SL(n, \C)$ or $\operatorname{SO}(n, \C)$. Let $H$ be a closed algebraic subgroup of $G$. The homogeneous space $G/H$ is called a \emph{spherical homogeneous space} if the action of a Borel subgroup on $G/H$ (by left multiplication) has a dense orbit. Spherical homogeneous spaces are generalizations of algebraic tori and, extending the theory of toric varieties, their compactifications can be described by means of combinatorial gadgets such as convex cones and convex polytopes. 

Let $(G/H)(\overline{\K})$ denote the Puiseux series valued points of $G/H$. One extends the tropicalization map from the torus case and constructs a \emph{spherical tropicalization map} $\operatorname{strop}: (G/H)(\overline{\K}) \to \mathcal{V}_{G/H}$. Here, $\mathcal{V}_{G/H}$ is a certain polyhedral cone associated to $G/H$ called the \emph{valuation cone} of $G/H$. Every point in the valuation cone $\mathcal{V}_{G/H}$ represents a discrete valuation (with values in $\R$) on the field of rational functions $\C(G/H)$ that is invariant under the left action of $G$ (see \cite[Section 4]{Vogiannou}, \cite[Section 5.1]{KM} and \cite{Nash}). 


Now consider the group $H = \GL(n, \C)$. The product group $G = \GL(n, \C) \times \GL(n, \C)$ acts on $\GL(n, \C)$ by  multiplication from left and right. That is, for $g_1, g_2, h \in \GL(n, \C)$:
$$(g_1, g_2) \cdot h := g_1 h  g_2^{-1}.$$
This action is obviously transitive. The stabilizer of the identity element is $H_{\diag} = \{(h, h) \mid h \in \GL(n, \C)\}$. Thus, $H = \GL(n, \C)$ can be identified with $(H \times H) / H_{\diag}$. One verifies that $H$ is a $(H \times H)$-spherical homogeneous space.

Following \cite[Theorem 2]{Vogiannou}, the tropicalization map in this case sends a matrix $A(t) \in \GL(n, \K)$ to its invariant factors,  ordered decreasingly (see the paragraph after Theorem \ref{th-Smith-normal-form}). We point out that, in general, it suffices to know the tropicalization map for $\K$-valued points only (see \cite[Lemma 14]{Vogiannou}). 

\section{Logarithm map and Archimedean amoebae}
\label{sec-amoeba}

For $t>0$, the {\it logarithm map} $\Log_t: (\C^*)^n \to \R^n$ is defined by:
\begin{equation} \label{equ-log-map-torus}
\Log_t(z_1, \ldots, z_n) = (\log_t|z_1|, \ldots, \log_t|z_n|).
\end{equation}
Clearly the inverse image of every point is an $(S^1)^n$-orbit in $(\C^*)^n$. Here $S^1$ denotes the unit circle and $(S^1)^n = \{(z_1, \ldots, z_n) \mid |z_1| = \cdots = |z_n| = 1\}$. It can be shown that $(S^1)^n$ is the largest compact subgroup of $(\C^*)^n$ (i.e. it is a \emph{maximal compact subgroup} of $(\C^*)^n$).
\begin{Def}
For a subvariety $Y \subset (\C^*)^n$, its \emph{(Archimedean) amoeba} $\mathcal{A}_t(Y)$ is the image of $Y$ in $\R^n$ under the logarithm map $\Log_t$.
\end{Def} 
Amoebas were introduced by Gelfand, Kapranov and Zelevinsky in \cite[Section 6.1]{GKZ}.
Figure \ref{fig-amoeba-line} shows an amoeba $\mathcal{A}_t(Y)$ of the line $Y$ in $(\C^*)^2$ given by the equation $x+y+1=0$, for some $0 < t < 1$. Note that this amoeba stretches to infinity along three directions (tentacles).

The amoeba $\mathcal{A}_t(Y)$ is a closed and unbounded subset of $\R^n$. 
An amoeba goes to infinity along certain directions usually called its {\it tentacles} (and hence the name amoeba). See \cite[Section 6.1]{GKZ}, \cite{Mikhalkin} and \cite[Section 1.4]{Maclagan-Sturmfels} for an introduction to the notion of amoeba.
\begin{figure}   
\includegraphics[width=7cm]{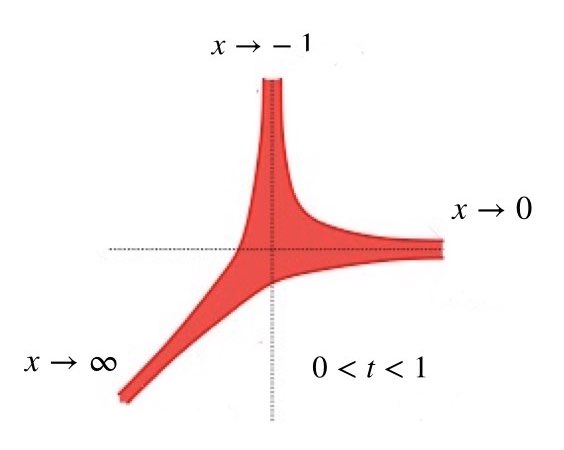} 
\caption{An amoeba of the line $x+y+1=0$ in $(\C^*)^2$.} 
\label{fig-amoeba-line} 
\end{figure} 
The directions along which an amoeba goes to infinity in fact coincides with the tropical variety of $Y$ and is one of the main motivations for introducing the notion of amoeba. This can be stated as follows:
\begin{center}
\emph{As $t \to 0$, the amoeba $\mathcal{A}_t(Y)$ approaches the tropical variety $\trop(Y)$.} 
\end{center}
When $Y$ is a hypersurface this is relatively easy to show and basically appears in \cite[Section 6.1, Proposition 1.9]{GKZ}. Even though the statement that, for arbitrary $Y$, the amoeba approaches the tropical variety has been known as a folklore, a precise formulation and proof only appeared relatively recently in (\cite[Theorem A]{Jonsson}). 

To make the statement precise we need to recall the notion of \emph{Kuratowski convergence} of subsets. It is an analogue of the Hausdorff convergence of compact subsets for not-necessarily-compact subsets (see \cite[Section 2]{Jonsson}). Let $B \subset \R$ and let $A_t$ be a family of subsets in $\R^n$ parameterized by points $t \in B$. Take $t_0 \in B$. One says that, as $t \to t_0$, the $A_t$ approach a subset $A \subset \R^n$ if the following holds:
\begin{itemize}
\item[(a)] For any $x \notin A$ there exists an open neighborhood $U \subset \R^n$ of $x$ and an open neighborhood $B_0 \subset B$ of $t_0$ such that $A_t \cap U = \emptyset$ for all $t \in B_0$.
\item[(b)] For any $x \in A$ and any open neighborhood $U \subset \R^n$ of $x$, there is an open neighborhood $B_0 \subset B$ of $t_0$ such that $A_t \cap U \neq \emptyset$, for all $t \in B_0$.
\end{itemize}
Roughly speaking, if, as $t \to t_0$, the $A_t$ approach $A$ in the sense of Kuratowski, then any point of $A$ is approached by points in the $A_t$ and no point outside $A$ is approached by points in the $A_t$.

\begin{Th}  \label{th-Jonsson}
As $t \to 0$ the amoeba $\mathcal{A}_t(Y)$ approaches the tropical variety $\trop(Y)$ (in the sense of Kuratowski convergence of subsets in $\R^n$). 
\end{Th}


In this paper, we suggest the logarithm of singular values of a matrix in $\GL(n, \C)$ as an analogue of the logarithm map on the torus $(\C^*)^n$ for $\GL(n, \C)$. 
The main theorem (Theorem \ref{th-main}) states that when $A(t) \in \GL(n, \K)$ is a matrix (over the formal Laurent series field) such that $A(t)$ is convergent for sufficiently small $t \neq 0$, then, as $t \to 0$, the logarithm of the set of  singular values of $A(t)$ approaches the set of invariant factors of $A$. This is an analogue of Theorem \cite[Theorem B]{Jonsson} for when the variety is a single point.

Let $Y \subset \GL(n, \C)$ be a subvariety. Let $Y(\overline{\K})$ be the Puiseux series valued points of $Y$. This means that each point in $Y(\overline{\K})$ is an invertible matrix whose entries are Puiseux series and they satisfy the defining polynomial equations of $Y$. We define the \emph{spherical amoeba} of $Y$ (to the base $t$) to be the following set: 
$$\{ (\log_t(d_1), \ldots, \log_t(d_n)) \mid d_1 \leq \cdots \leq d_n \text{ are singular values of some matrix } A \in Y\}.$$
Following \cite[Theorem 2]{Vogiannou}, we also define the \emph{spherical tropical variety} of $Y$ to be the (topological) closure of the set of invariant factors (ordered decreasingly) of all matrices $A(t) \in Y(\overline{\K})$, i.e., the entries of $A(t)$ are Puiseux series that satisfy the defining polynomial equations of $Y$. We make the following conjecture (see also \cite[Section 7]{KM}). 
\begin{Conj}  \label{conj-sph-amoeba}
As $t \to 0$, the spherical amoeba of $Y$ approaches the spherical tropical variety of $Y$ (in the sense of Kuratowski).  
\end{Conj}


\section{Some background material from linear algebra}
In this section we review some well-known statements from linear algebra that appear in the statement or proof of our main theorem (Theorem \ref{th-main}). The first one is the \emph{singular value decomposition} (for a proof see for example \cite[Chapter 8, p. 170]{Lax}). It has many applications in mathematics, engineering and statistics.  

\begin{Th}[Singular value decomposition]  \label{th-SVD}
Let $A \in \Mat(n, \C)$. Then $A$ can be written as: $$A = UDW,$$ where $U$, $W$ are $n \times n$ unitary matrices and $D$ is diagonal with nonnegative real entries. Moreover, the diagonal entries of $D$ are the eigenvalues of the positive semi-definite matrix $\sqrt{AA^*}$ where $A^* = \bar{A}^t$. The diagonal entries of $D$ are called the {\it singular values of $A$}.
\end{Th}

We point out that the singular value decomposition holds for non-square matrices as well.

There is an analogue of the singular value decomposition for matrices over the field of fractions of a principal ideal domain (for a proof see for example \cite[Section 12.1, Theorem 5]{DF}). 

\begin{Th}[Smith normal form]  \label{th-Smith-normal-form} 
Let $\mathcal{O}$ be a principal ideal domain with field of fractions $K$. Consider a matrix $A \in \Mat(n, K)$. 
Then $A$ can be written as: $$A = P \tau Q,$$ where $P$, $Q \in \GL(n, \mathcal{O})$ and $\tau$ is a diagonal matrix $\diag(a_1, \ldots, a_n)$ such that $a_1 | a_2 | \cdots | a_n$ (for $a, b \in K$ we say $a | b$ if $b/a \in \mathcal{O}$). The diagonal entries of $\tau$ are called the {\it invariant factors of $A$}.

\end{Th}

Let $A(t)$ be an $n \times n$ matrix whose entires $A_{ij}(t)$ are Laurent series in $t$ over $\C$. Since the ring of formal power series $\C[[t]]$ is a PID, the Smith normal form theorem applies to $A(t)$ and states that it can be written as:
$$P(t) \tau(t) Q(t),$$
where $P(t), Q(t)$ are $n \times n$ matrices with power series entries such that their determinants are invertible power series, and $\tau(t)$ is a diagonal matrix. Every element $f \in \K = \C((t))$ can be written as $f = t^v h$ where $v \in \Z$ and $h \in \mathcal{O}=\C[[t]]$ is a power series with non-zero constant term and hence a unit in $\mathcal{O}$. Thus, after multiplying with a diagonal matrix in $\GL(n, \mathcal{O})$ we can assume that
$\tau(t) = \diag(t^{v_n}, \ldots, t^{v_1})$ with integers $v_n \leq  \cdots \leq v_1$. By slight abuse of terminology, we refer to $v_1, \ldots, v_n$ as the {\it invariant factors of $A(t)$}. 

Finally we recall the Hilbert-Courant min-max principle about the eigenvalues of a Hermitian matrix (see \cite[p. 116, Theorem 10]{Lax}). 
 
\begin{Th}[Min-max principle]    \label{th-Hilbert-Courant-minmax}
Let $H$ be an $n \times n$ Hermitian matrix with eigenvalues $\lambda_1 \leq \cdots \leq \lambda_n$. Then:
\begin{equation}  \label{equ-min-max}
\lambda_k = \min \Bigl \{ \max_{x \in W \setminus \{0\}} \frac{(Hx, x)}{(x, x)} \mid W \subset \C^n \textup{ linear subspace with }\dim(W) = k \Bigr \},
\end{equation}
and
\begin{equation}    \label{equ-max-min}
\lambda_k = \max \Bigl \{ \min_{x \in U \setminus \{0\} } \frac{(Hx, x)}{(x, x)} \mid U \subset \C^n \textup{ linear subspace with } \dim(U) = n-k+1 \Bigr \}.
\end{equation}
In the above, $(x, y) = \sum_{i=1}^n x_i \bar{y}_i$ denotes the standard Hermitian product in $\C^n$.
\end{Th}

\section{Proof of main result}    \label{sec-main}
We now give a proof of Theorem \ref{th-main}. We will use the following lemma.
\begin{Lem} \label{lem-m-M}
Let $P(t) \in \GL(n, \mathcal{O})$ where $\mathcal{O} = \C[[t]]$. Suppose $P(t)$ is convergent for sufficiently small values of $t$, that is, there exists $\epsilon > 0$ such that all the power series $P_{ij}(t)$ are convergent for any $0 \leq |t| < \epsilon$. Then there are constants $0 < m < M$ such that for sufficiently small $t$ and any $x \in \C^n$ with $||x||=1$, we have $m < ||P(t)x|| < M$. Here $||x|| = (\sum_{i=1}^n x_i \bar{x}_i)^{1/2}$.
\end{Lem}
\begin{proof}
Recall that the operator norm of a matrix $P \in \Mat(n, \C)$ is defined by:
$$||P|| = \max\{ || P x || \mid x \in \C^n,~ ||x|| =1 \}.$$
One knows that the operator norm depends continuously on the entries of $P$. Under the assumptions in the lemma, the entries of $P(t)$ depend continuously on $t$ and hence as $t$ varies in a closed bounded interval containing $0$, the entries of $P(t)$ and hence $||P(t)||$ remain bounded above. To show the lower bound, suppose by contradiction that there is a sequence $\{t_i\}$, $\lim_{i \to \infty} t_i = 0$, and a sequence $\{x_i\}$, $||x_i|| = 1$, such that $\lim_{i \to \infty} P(t_i)x_i = 0$. After passing to a subsequence, we can assume $\{x_i\}$ is convergent to some unit vector $x$. It follows that $P(0)x = 0$. But this contradicts the assumption that $P(t) \in \GL(n, \mathcal{O})$ and hence $\det(P(0)) \neq 0$.
\end{proof}

\begin{proof}[Proof of Theorem \ref{th-main}]
Throughout the proof we write $A_t$ in place of $A(t)$. We first show that for any $1 \leq k \leq n$ we have: 
$$\lim_{t \to 0^+} \log_t d_k(t) \leq v_k.$$

Let $\mathcal{O}' \subset \mathcal{O}=\C[[t]]$ denote the subring of power series with nonzero radius of convergence. It is straightforward to see that $\mathcal{O}'$ is a discrete valuation ring as well (see, for instance, \cite[Section 8.1]{Reid}). The field of fractions $\mathcal{K}'$ of $\mathcal{O}'$ is the subfield of formal Laurent series $\C((t))$ that are convergent in a punctured neighborhood of the origin. Thus, $A_t \in \GL(n, \mathcal{K}')$. By the Smith normal form theorem, applied to the discrete valuation ring $\mathcal{O}'$, we can write
$$A_t = P_t \tau_t Q_t,$$ where $\tau_t$ is a diagonal matrix $\diag(t^{v_1}, \ldots, t^{v_n})$ and $P_t$, $Q_t \in \GL(n, \mathcal{O}')$. Thus, the entries of the matrices $P_t$ and $Q_t$ are convergent in a neighborhood of the origin.

The $k$-th singular value $d_k(t)$ of $A_t$ is the square root of the $k$-th eigenvalue of $A^*_t A_t$. We note that for any $x \in \C^n$, $(A^*_t A_t x, x) = (A_tx, A_tx)$. By \eqref{equ-min-max} in Theorem \ref{th-Hilbert-Courant-minmax}, we then have:
$$d_k(t)^2 = \min \{ \max \{ \frac{(A_tx, A_tx)}{(x, x)} \mid 0 \neq x \in W\} \mid \dim(W) = k \}.$$ 

Fix a subspace $W \subset \C^n$ with $\dim(W) = k$. Then:
$$d_k(t)^2 \leq \max \{ \frac{(A_tx, A_tx)}{(x, x)} \mid 0 \neq x \in W\}.$$ 

Let $y_t = Q_t x$ and $z_t = \tau_t Q_t x$. Then: $$\frac{(A_tx, A_tx)}{(x, x)} = \frac{(P_t z_t, P_t z_t)}{(z_t, z_t)} \cdot \frac{(\tau_t y_t, \tau_t y_t)}{(y_t, y_t)} \cdot \frac{(Q_t x, Q_tx)}{(x, x)}.$$ Since we work with non-negative numbers, the maximum of a product is less than or equal to the product of maximums. Hence we have:
\begin{multline}   \label{equ-d_k-max}
d_k(t)^2 \leq  \max \{ \frac{(P_tz, P_tz)}{(z, z)} \mid 0 \neq z \in \tau_tQ_t(W)\} ~\cdot \\
\cdot \max \{ \frac{(\tau_t y, \tau_t y)}{(y, y)} \mid 0 \neq y \in Q_t(W)\} \cdot
\max \{ \frac{(Q_t x, Q_tx)}{(x, x)} \mid 0 \neq x \in W\}.
\end{multline}
Now take $W$ to be the subspace $Q_t^{-1}(E_k)$ where $E_k = \operatorname{span}\{ e_1, \ldots, e_k\}$, the span of the first $k$ standard basis elements. With this choice and for $0<t<1$ we have:
\begin{equation}  \label{equ-tau_t-max}
t^{2v_k} = \max \{ \frac{(\tau_t y, \tau_t y)}{(y, y)} \mid 0 \neq y \in Q_t(W)=E_k\}.
\end{equation}
By Lemma \ref{lem-m-M}, we can find $M, M' > 0$ such that for all $0 \neq x, z \in \C^n$ we have:
$$\frac{(P_tz, P_tz)}{(z, z)} < M,$$
$$\frac{(Q_tx, Q_tx)}{(x, x)} < M'.$$
In view of \eqref{equ-d_k-max} and \eqref{equ-tau_t-max}, this implies that $d_k(t)^2 \leq M \cdot t^{2v_k} \cdot M'$. After taking $\log_t$, with $0 < t < 1$, we obtain:
\begin{equation}   \label{equ-log-d_k-geq-v_k}
2\log_t(d_k(t) \geq \log_t(M) + 2v_k + \log_t(M').
\end{equation}
To prove the reverse inequality, we apply \eqref{equ-max-min} in Theorem \ref{th-Hilbert-Courant-minmax} to the Hermitian matrix $A^*_t A_t$ to get:
\begin{equation} \label{equ-d_k-min}
d_k(t)^2 = \max \{ \min \{ \frac{(A_tx, A_tx)}{(x, x)} \mid 0 \neq x \in U\} \mid \dim(U) = n-k+1 \}.
\end{equation}
If we fix a subspace $U$ with $\dim(U) = n-k+1$, from \eqref{equ-d_k-min} we have: $$d_k(t)^2 \geq \min \{ \frac{(A_tx, A_tx)}{(x, x)} \mid 0 \neq x \in U\}.$$ 
Again, since we work with non-negative numbers, the minimum of a product is greater than or equal to the product of minimums. Thus, with the notation as before, we have:
\begin{multline}   \label{equ-d_k-min2}
d_k(t)^2 \geq  \min \{ \frac{(P_tz, P_tz)}{(z, z)} \mid 0 \neq z \in \tau_tQ_t(U)\}  \cdot \\
\cdot \min \{ \frac{(\tau_t y, \tau_t y)}{(y, y)} \mid 0 \neq y \in Q_t(U)\} \cdot
\min \{ \frac{(Q_t x, Q_tx)}{(x, x)} \mid 0 \neq x \in U\}.
\end{multline}
Now take $U$ to be the subspace $Q_t^{-1}(E'_k)$ where $E'_k = \operatorname{span}\{ e_k, \ldots, e_n\}$, the space of last $n-k+1$ standard basis elements. With this choice we have:
\begin{equation}  \label{equ-tau_t-min}
t^{2v_k} = \min \{ \frac{(\tau_t y, \tau_t y)}{(y, y)} \mid 0 \neq y \in Q_t(U)=E_k'\}.
\end{equation}
By Lemma \ref{lem-m-M} we can find $m, m' > 0$ such that for all $0 \neq x, z \in \C^n$ we have $(P_tz, P_tz)/(z, z) > m,$ and $(Q_tx, Q_tx)/(x, x) > m'.$
In view of \eqref{equ-d_k-min2} and \eqref{equ-tau_t-min}, this implies that $d_k(t)^2 \geq m \cdot t^{2v_k} \cdot m'$ and after taking $\log_t$, with $0 < t < 1$, we obtain:
\begin{equation}   \label{equ-log-d_k-leq-v_k}
2 \log_t(d_k(t)) \leq \log_t(m) + 2v_k + \log_t(m').
\end{equation}
Taking limit $t \to 0$ in the inequalities \eqref{equ-log-d_k-geq-v_k} and \eqref{equ-log-d_k-leq-v_k} we obtain $$\lim_{t \to 0}  \log_t(d_k(t)) = v_k,$$ as required. 
\end{proof}

\end{document}